\documentclass[12pt]{amsart}
\usepackage{amsmath,amsthm,amsfonts,amssymb,mathrsfs}
\date{\today}

\usepackage{color}

\usepackage{hyperref}

\newtheorem{theorem}{Theorem}[section]

\newtheorem{proposition}[theorem]{Proposition}
\newtheorem{corollary}[theorem]{Corollary}

\theoremstyle{definition}
\newtheorem{example}[theorem]{Example}
\newtheorem{remark}[theorem]{Remark}

\newtheorem{definition}[theorem]{Definition}

\begin{document}

\title[On closures in semitopological inverse semigroups with continuous inversion]{On closures in semitopological inverse semigroups with continuous inversion}

\author{Oleg~Gutik}
\address{Faculty of Mechanics and Mathematics, National University of Lviv,
Universytetska 1, Lviv, 79000, Ukraine}
\email{o\underline{\hskip5pt}\,gutik@franko.lviv.ua,
ovgutik@yahoo.com}



\begin{abstract}
We study the closures of subgroups, semilattices and different kinds of semigroup extensions in semitopological inverse semigroups with continuous inversion. In particularly we show that a topological group $G$ is $H$-closed in the class of semitopological inverse semigroups with continuous inversion if and only if $G$ is compact, a Hausdorff linearly ordered topological semilattice $E$ is $H$-closed in the class of semitopological semilattices if and only if $E$ is $H$-closed in the class of topological semilattices, and a topological Brandt $\lambda^0$-extension of $S$ is (absolutely) $H$-closed in the class of semitopological inverse semigroups with continuous inversion if and only if so is $S$. Also, we construct an example of an $H$-closed non-absolutely $H$-closed semitopological semilattice in the class of semitopological semilattices.
\end{abstract}

\maketitle

\section{Introduction and preliminaries}

We shall follow the terminology of \cite{ArhangelskiiTkachenko2008, CliffordPreston1961-1967, GierzHofmannKeimelLawsonMisloveScott2003, Petrich1984, Ruppert1984}.

A subset $A$ of an infinite set $X$ is called \emph{cofinite in $X$} if $X\setminus A$ is finite.

Given a semigroup $S$, we shall denote the set of idempotents of $S$ by $E(S)$. A \emph{semilattice} is a commutative semigroup of idempotents. For a semilattice $E$  the semilattice operation on $E$ determines the partial order $\leqslant$ on $E$: $$e\leqslant f\quad\text{if and only if}\quad ef=fe=e.$$ This order is called {\em natural}. An element $e$ of a partially ordered set $X$ is
called {\em minimal} if $f\leqslant e$  implies $f=e$ for $f\in X$. An idempotent $e$ of a semigroup $S$ without zero (with zero $0_S$) is called \emph{primitive} if $e$ is a minimal element in $E(S)$ (in $(E(S))\setminus\{0_S\}$).  A \emph{maximal chain} of a semilattice $E$ is a chain which is properly contained in no other chain of $E$. The Axiom of Choice implies the existence of maximal chains in any partially ordered set.

A semigroup $S$ with the adjoined unit [zero] will be denoted by $S^1$ [$S^0$] (cf. \cite{CliffordPreston1961-1967}). Next, we shall denote the unit (identity) and the zero of a semigroup $S$ by $1_S$ and $0_S$, respectively. Given a subset $A$ of a semigroup $S$, we shall denote by $A^*=A\setminus\{ 0_S\}$ and $|A|=$ the cardinality of $A$. A semigroup $S$ is called \emph{inverse} if for any $x\in S$ there exists a unique $y\in S$ such that $xyx=x$ and $yxy=y$. Such an element $y$ is called \emph{inverse} of $x$ and it is denoted by $x^{-1}$.

If $h\colon S\rightarrow T$ is a homomorphism (or a map) from a semigroup $S$ into a semigroup $T$ and if $s\in S$, then we denote the image of $s$ under $h$ by $(s)h$. A semigroup homomorphism $h\colon S\rightarrow T$ is called \emph{annihilating} if $(s)h=(t)h$ for all $s, t\in S$.

Let $S$ be a semigroup with zero and $\lambda$ a cardinal $\geqslant 1$. We define the semigroup operation on the set $B_{\lambda}(S)=(\lambda\times S\times {\lambda})\cup\{ 0\}$ as follows:
\begin{equation*}
 (\alpha,a,\beta)\cdot(\gamma, b, \delta)=
  \begin{cases}
    (\alpha, ab, \delta), & \text{ if } \beta=\gamma; \\
    0, & \text{ if } \beta\ne \gamma,
  \end{cases}
\end{equation*}
and $(\alpha, a, \beta)\cdot 0=0\cdot(\alpha, a, \beta)=0\cdot
0=0,$ for all $\alpha, \beta, \gamma, \delta\in {\lambda}$ and $a,
b\in S$. If $S=S^1$ then the semigroup $B_\lambda(S)$ is called
the {\it Brandt $\lambda$-extension of the semigroup}
$S$~\cite{Gutik1999}. Obviously, if $S$ has zero
then ${\mathcal J}=\{ 0\}\cup\{(\alpha, 0_S, \beta)\mid 0_S$ is
the zero of $S\}$ is an ideal of $B_\lambda(S)$. We put
$B^0_\lambda(S)=B_\lambda(S)/{\mathcal J}$ and the semigroup
$B^0_\lambda(S)$ is called the {\it Brandt $\lambda^0$-extension
of the semigroup $S$ with zero}~\cite{GutikPavlyk2006}.

Next, if
$A\subseteq S$ then we shall denote $A_{\alpha\beta}=\{(\alpha, s,
\beta)\mid s\in A \}$ if $A$ does not contain zero, and
$A_{\alpha,\beta}=\left\{(\alpha, s, \beta)\mid s\in A\setminus\{ 0\} \right\}\cup \{ 0\}$ if $0\in A$, for $\alpha, \beta\in {\lambda}$.

We shall denote the semigroup of ${\lambda}{\times}{\lambda}$-matrix units by $B_\lambda$ and the subsemigroup of ${\lambda}\times{\lambda}$-matrix units of the Brandt $\lambda^0$-extension of a monoid $S$ with zero by $B^0_\lambda(1)$. We always consider the Brandt $\lambda^0$-extension only of a monoid with zero. Obviously, for
any monoid $S$ with zero we have $B^0_1(S)=S$. Note that every Brandt $\lambda$-extension of a group $G$ is isomorphic to the Brandt $\lambda^0$-extension of the group $G^0$ with adjoined zero. The Brandt $\lambda^0$-extension of the group with adjoined zero is called a \emph{Brandt  semigroup}~\cite{CliffordPreston1961-1967, Petrich1984}. A semigroup $S$ is a Brandt semigroup if and only if $S$ is a completely $0$-simple inverse semigroup~\cite{Clifford1942, Munn1957} (cf.  also \cite[Theorem~II.3.5]{Petrich1984}). We also observe that the semigroup $B_\lambda$ of ${\lambda \times\lambda}$-matrix units is isomorphic to the Brandt $\lambda^0$-extension of the two-element monoid with zero $S=\{ 1_S, 0_S\}$ and the
trivial semigroup $S$ (i.~e. $S$ is a singleton set) is isomorphic to the Brandt $\lambda^0$-extension of $S$ for every cardinal $\lambda\geqslant 1$.

Let $\left\{S_{\iota}\colon \iota\in\mathscr{I}\right\}$ be a disjoint family of semigroups with zero such that $0_\iota$ is zero in $S_{\iota}$ for any $\iota\in\mathscr{I}$. We put $S=\{0\}\cup\bigcup\left\{S_{\iota}^*\colon \iota\in\mathscr{I}\right\}$, where $0\notin\bigcup\left\{S_{\iota}^*\colon \iota\in\mathscr{I}\right\}$, and define a semigroup  operation ``\;$\cdot$\;'' on $S$ in the following way
\begin{equation*}
    s\cdot t=
\left\{
  \begin{array}{cl}
    st, & \hbox{if } st\in S_{\iota}^* \hbox{ for some } \iota\in\mathscr{I};\\
    0, & \hbox{otherwise}.
  \end{array}
\right.
\end{equation*}
The semigroup $S$ with the operation ``\;$\cdot$\;'' is called an \emph{orthogonal sum} of the semigroups $\left\{S_{\iota}\colon \iota\in\mathscr{I}\right\}$ and in this case we shall write $S=\sum_{\iota\in\mathscr{I}}S_{\iota}$.

A non-trivial inverse semigroup is called a \emph{primitive inverse semigroup} if all its non-zero idempotents are primitive~\cite{Petrich1984}. A semigroup $S$ is a primitive inverse semigroup if and only if $S$ is an orthogonal sum of Brandt semigroups~\cite[Theorem~II.4.3]{Petrich1984}.

In this paper all topological spaces are Hausdorff. If $Y$ is a subspace of a topological space $X$ and $A\subseteq Y$, then by $\operatorname{cl}_Y(A)$ we denote the topological closure of $A$ in $Y$.

A ({\it semi})\emph{topological semigroup} is a Hausdorff topological space with a (separately) continuous semigroup operation. A topological semigroup which is an inverse semigroup is called an \emph{inverse topological semigroup}. A \emph{topological inverse semigroup} is an inverse topological semigroup with continuous inversion. We observe that the inversion on a (semi)topological inverse semigroup is a homeomorphism (see \cite[Proposition~II.1]{EberhartSelden1969}). A \emph{semitopological group} is a Hausdorff topological space with a separately continuous group operation. A semitopological group with continuous inversion is a \emph{quasitopological group}. A \emph{paratopological group} is called a group with a continuous group operation. A paratopological group with continuous inversion is a \emph{topological group}.

Let $\mathfrak{STSG}_0$ be a class of semitopological semigroups. A semigroup $S\in\mathfrak{STSG}_0$ is called {\it $H$-closed in} $\mathfrak{STSG}_0$, if $S$ is a closed subsemigroup of any topological semigroup $T\in\mathfrak{STSG}_0$ which contains $S$ both as a subsemigroup and as a topological space. The $H$-closed
topological semigroups were introduced by Stepp in \cite{Stepp1969}, and there they were called {\it maximal semigroups}. A semitopological semigroup $S\in\mathfrak{STSG}_0$ is called {\it absolutely $H$-closed in the class} $\mathfrak{STSG}_0$, if any continuous homomorphic image of $S$ into $T\in\mathfrak{STSG}_0$ is $H$-closed in $\mathfrak{STSG}_0$. An algebraic semigroup $S$ is called:
\begin{itemize}
  \item {\it algebraically complete in} $\mathfrak{STSG}_0$, if $S$ with any Hausdorff topology $\tau$ such that $(S,\tau)\in\mathfrak{STSG}_0$ is $H$-closed in $\mathfrak{STSG}_0$;
  \item {\it algebraically $h$-complete in} $\mathfrak{STSG}_0$, if $S$
with discrete topology ${\mathfrak{d}}$ is absolutely $H$-closed in $\mathfrak{STSG}_0$ and $(S,{\mathfrak{d}})\in\mathfrak{STSG}_0$.
\end{itemize}
Absolutely $H$-closed topological semigroups and algebraically $h$-complete semigroups were introduced by Stepp in~\cite{Stepp1975}, and there they were called {\it absolutely maximal} and {\it algebraic maximal}, respectively.

Recall \cite{Aleksandrov1942}, a topological group $G$ is called {\it absolutely
closed} if $G$ is a closed subgroup of any topological group which contains $G$ as a subgroup. In our terminology such topological groups are called $H$-closed in the class of topological groups. In \cite{Raikov1946} Raikov proved that a topological group $G$ is absolutely closed if and only if it is Raikov complete, i.e. $G$
is complete with respect to the two-sided uniformity. A topological group $G$ is called {\it $h$-complete} if for every continuous homomorphism $h\colon G\to H$ the subgroup $f(G)$ of $H$ is closed~\cite{DikranjanUspenskij1998}. In our terminology such topological groups are called absolutely $H$-closed in the class of topological groups. The $h$-completeness is preserved under taking products and closed central subgroups~\cite{DikranjanUspenskij1998}. $H$-closed paratopological and topological groups in the class of paratopological groups studied in \cite{Ravsky2003}.

In \cite{Stepp1975} Stepp studied $H$-closed topological semilattice in the class of topological semigroups. There he proved that an algebraic semilattice $E$ is algebraically $h$-complete in the class of topological semilattices if and only if every chain in $E$ is finite. In \cite{GutikRepovs2008} Gutik and Repov\v{s} established the closure of a linearly ordered topological semilattice in a topological semilattice. They proved the criterium of $H$-closedness of a linearly ordered topological semilattice in the class of topological semilattices and showed that every $H$-closed topological semilattice is absolutely $H$-closed in the class of topological semilattices. Also, such semilattices studied in \cite{ChuchmanGutik2007, GutikPagonRepovs2010}. In \cite{BardylaGutik2012} the structure of closures of the discrete semilattices $(\mathbb{N},\min)$ and $(\mathbb{N},\max)$ is described. Here the authors constructed an example of an $H$-closed topological semilattice in the class of topological semilattices which is not absolutely $H$-closed in the class of topological semilattices. The constructed example gives a negative answer on Question~17 from \cite{Stepp1975}.

\begin{definition}[\cite{GutikPavlyk2006}]\label{def2}
Let $\mathfrak{STSG}_0$ be a class of semitopological semigroups.
Let $\lambda\geqslant 1$ be a cardinal and $(S,\tau)\in\mathfrak{STSG}_0$. Let $\tau_{B}$ be a topology on $B^0_{\lambda}(S)$ such that
\begin{itemize}
  \item[a)] $\left(B^0_{\lambda}(S),            \tau_{B}\right)\in\mathfrak{STSG}_0$;
  \item[b)] the topological subspace $(S_{\alpha,\alpha},\tau_{B}|_{S_{\alpha,\alpha}})$ is naturally homeomorphic to $(S,\tau)$ for some $\alpha\in{\lambda}$.
\end{itemize}
Then $\left(B^0_{\lambda}(S), \tau_{B}\right)$ is called a {\it topological Brandt $\lambda^0$-extension of $(S, \tau)$ in $\mathfrak{STSG}_0$}.
\end{definition}

In the paper \cite{GutikRepovs2010} Gutik and Repov\v{s} established homomorphisms of the Brandt $\lambda^0$-extensions of monoids with zeros. They also described a category whose objects are ingredients in the constructions of the Brandt $\lambda^0$-extensions of monoids with zeros. Here they introduced finite, compact topological Brandt $\lambda^0$-extensions of topological semigroups and countably compact topological Brandt $\lambda^0$-extensions of topological inverse semigroups in the class of topological inverse semigroups, and established the structure of such extensions and non-trivial continuous homomorphisms between such topological Brandt $\lambda^0$-extensions of topological  monoids with zero. There they also described a category whose objects are ingredients in the constructions of finite (compact,
countably compact) topological Brandt $\lambda^0$-extensions of topological  monoids with zeros. These  investigations were continued in \cite{GutikPavlyk2013a, GutikPavlykReiter2009, GutikRavsky20??}, where established countably compact topological Brandt $\lambda^0$-extensions of topological monoids with zeros and pseudocompact topological Brandt $\lambda^0$-extensions of semitopological monoids with zeros their corresponding categories. In the papers \cite{BerezovskiGutikPavlyk2010, GutikPavlyk2001, GutikPavlyk2003, GutikPavlyk2006, Pavlyk2004} were studied $H$-closed and absolutely $H$-closed topological Brandt $\lambda^0$-extensions of topological semigroups in the class of topological semigroups.

In Section~\ref{s2} we study the closure of a quasitopological group in a semitopological inverse semigroup with continuous inversion. In particularly we show that a topological group $G$ is $H$-closed in the class of semitopological inverse semigroups with continuous inversion if and only if $G$ is compact.

Section~\ref{s3} is devoted to the closure of a semitopological semilattice in a semitopological inverse semigroup with continuous inversion. We show that a Hausdorff linearly ordered topological semilattice $E$ is $H$-closed in the class of semitopological semilattices if and only if $E$ is $H$-closed in the class of topological semilattices. Also, we construct an example of an $H$-closed semitopological semilattice in the class of semitopological semilattices which is not absolutely $H$-closed in the class of semitopological semilattices.

In Section~\ref{s4} we show that a topological Brandt $\lambda^0$-extension of $S$ is (absolutely) $H$-closed in the class of semitopological inverse semigroups with continuous inversion if and only if so is $S$. Also, we study the preserving of (absolute) $H$-closedness in the class of semitopological inverse semigroups with continuous inversion by orthogonal sums.


\section{On the closure of a quasitopological group in a semitopological inverse semigroup with continuous inversion}\label{s2}

\begin{proposition}\label{proposition-2.1}
Every left topological inverse semigroup with continuous inversion is semitopological semigroup.
\end{proposition}

\begin{proof}
We write an arbitrary right translation $\rho_a\colon S\to S\colon x\mapsto xa$ of a left topological inverse semigroup $S$ with continuous inversion $\operatorname{\textbf{inv}}\colon S\to S$ on three steps in the following way:
\begin{equation*}
    \rho_a(x)=xa=\left(a^{-1}x^{-1}\right)^{-1}= \left(\operatorname{\textbf{inv}}\circ\lambda_{a^{-1}}\circ \operatorname{\textbf{inv}}\right)(x).
\end{equation*}
This implies the continuity of right translations i $S$.
\end{proof}

It is well known that the closure of an inverse subsemigroup of a topological inverse semigroup is again a topological inverse semigroup (see: \cite[Proposition~II.1]{EberhartSelden1969}). The following proposition extends this result to semitopological inverse semigroups with continuous inversion.

\begin{proposition}\label{proposition-2.2}
The closure of an inverse subsemigroup $T$ in a semitopological inverse semigroup $S$ with continuous inversion is an inverse semigroup.
\end{proposition}

\begin{proof}
By Proposition~1.8$(ii)$ from \cite[Chapter~I, Proposition~1.8$(ii)$]{Ruppert1984} the closure $\operatorname{cl}_S(T)$ of $T$ in a semitopological semigroup $S$ is a semitopological semigroup. Then the continuity of the inversion $\operatorname{\textbf{inv}}\colon S\to S$ and Theorem~1.4.1 from \cite{Engelking1989} imply that $\operatorname{\textbf{inv}}(\operatorname{cl}_S(T))\subseteq \operatorname{cl}_S(\operatorname{\textbf{inv}}(T))= \operatorname{cl}_S(T)$ and hence we get that $\operatorname{\textbf{inv}}(\operatorname{cl}_S(T))= \operatorname{cl}_S(T)$. This implies that $\operatorname{cl}_S(T)$ is an inverse subsemigroup of $S$.
\end{proof}

We observe that the statement of Proposition~\ref{proposition-2.2} is not true in the case of inverse topological semigroup. It is complete to consider the set $\mathbb{R}^+=[0,+\infty)$ of non-negative real numbers with usual topology and usual multiplication of real numbers. This implies that in Proposition~\ref{proposition-2.2} the condition that $S$ has continuous inversion is essential.
	
In a compact topological semigroup the closure of a subgroup is a topological subgroup (see: \cite[Vol.~1, Theorems~1.11 and 1.13]{CarruthHildebrantKoch}). Also, since for a topological inverse semigroup $S$ the map $f\colon S\to S\colon x\rightarrow xx^{-1}$ is continuous, the maximal subgroup of $S$ is closed, and hence the closure of a subgroup of a topological inverse semigroup is a subgroup. The previous observation implies that this is not true in the general case of topological semigroups. Also, the following example shows that the closure of a subgroup in a semitopological inverse semigroup with continuous inversion is not a subgroup.

\begin{example}\label{example-2.3}
Let $\mathbb{Z}$ be the discrete additive group of integers. We put $\mathscr{A}(\mathbb{Z})$ is the one point Alexandroff compactification of the space $\mathbb{Z}$ with the remainder $\infty$. We extend the semigroup operation from $\mathbb{Z}$ onto $\mathscr{A}(\mathbb{Z})$ in the following way:
\begin{equation*}
    n+\infty=\infty+n=\infty+\infty=\infty, \qquad \mbox{for every } \; n\in \mathbb{Z}.
\end{equation*}
It is well known that $\mathscr{A}(\mathbb{Z})$ with such defined operation is a semitopological inverse semigroup with continuous inversion and $\mathbb{Z}$ is not a closed subgroup of $\mathscr{A}(\mathbb{Z})$ \cite{Ruppert1984}.
\end{example}

A quasitopological group $G$ is called \emph{precompact} if for every open neighbourhood $U$ of the neutral element of $G$ there exists a finite subset $F$ of $G$ such that $UF=G$ \cite{ArhangelskiiTkachenko2008}.

The following proposition gives examples quasitopological groups which are non-closed subgroups of some semitopological inverse semigroups with continuous inversion.

\begin{proposition}\label{proposition-2.4}
For every non-precompact regular quasitopological group $(G,\tau)$ there exists a regular semitopological inverse semigroup with continuous inversion which contains $(G,\tau)$ as a non-closed subgroup.
\end{proposition}

\begin{proof}
Since the quasitopological group $(G,\tau)$ is non-precompact there exists an open neighbourhood $U$ of the neutral element $e$ of the group $G$ such that $FU\neq G$ and $UF\neq G$ for every finite subset $F$ in $G$. Let $\mathscr{B}_e$ be a base of the topology $\tau$ at the neutral element $e$ of $(G,\tau)$. Since the inversion is continuous in $(G,\tau)$, without loss of generality we may assume that all elements of the family $\mathscr{B}_e$ are symmetric, i.e., $V=V^{-1}$ for every $V\in\mathscr{B}_e$. We put
\begin{equation*}
    \mathscr{B}_U=\left\{ V\in\mathscr{B}_e\colon \operatorname{cl}_G(V)\subseteq U\right\}.
\end{equation*}
Since the quasitopological group $(G,\tau)$ is not precompact we have that $FV\neq G$ and $VF\neq G$ for every $V\in\mathscr{B}_U$ and for every finite subset $F$ in $G$.

By $G^0$ we denote the group $G$ with a joined zero $0$. Now, we put
\begin{equation*}
    \mathscr{P}_0=\left\{W_{g,V}=\{0\}\cup G\setminus\operatorname{cl}_G(gV)\colon V\in\mathscr{B}_U, g\in G \right\}\cup \left\{W_{V,g}=\{0\}\cup G\setminus\operatorname{cl}_G(Vg)\colon V\in\mathscr{B}_U, g\in G \right\}
\end{equation*}
and $\tau\cup\mathscr{P}_0$ is a subbase of a topology $\tau_0$ on $G^0$.

Since $(G,\tau)$ a quasitopological group, it is sufficient to show that the semigroup operation on $(G^0,\tau_0)$ is separately continuous in the following two cases: $h\cdot 0=0$ and $0\cdot h=0$, for $h\in G$. Then for arbitrary subbase neighbourhoods $W_{g_1,V_1},\ldots,W_{g_n,V_n}$ and $W_{V_1,g_1},\ldots,W_{V_n,g_n}$ we have that
\begin{equation*}
    h\cdot \left(W_{g_1,V_1}\cap\cdots\cap W_{g_n,V_n}\right)\subseteq W_{hg_1,V_1}\cap\cdots\cap W_{hg_n,V_n}
\end{equation*}
and
\begin{equation*}
    \left(W_{V_1,g_1}\cap\cdots\cap W_{V_n,g_n}\right)\cdot h\subseteq W_{V_1,g_1h}\cap\cdots\cap W_{V_n,g_nh}.
\end{equation*}
Also, since translations in the quasitopological group $(G,\tau)$ are homeomorphisms, for every open subbase neighbourhood $V\in\mathscr{B}_U$ of the neutral element of $G$ and every $g\in G$ we have that $\left(W_{g,V}\right)^{-1}\subseteq W_{V^{-1},g^{-1}}$. Therefore $(G^0,\tau_0)$ is a quasitopological inverse semigroup with continuous inversion.

Now for every open subbase neighbourhoods $V_1,V_2\in\mathscr{B}_U$ of the neutral element of $G$ such that $\operatorname{cl}_G(V_1)\subseteq V_2$ and every $g\in G$ the following conditions holds:
\begin{equation*}
\operatorname{cl}_G(W_{g,V_2})\subseteq W_{g,V_1} \qquad \mbox{and} \qquad
\operatorname{cl}_G(W_{V_2,g})\subseteq W_{V_1,g}.
\end{equation*}
Hence we get that the topological space $(G^0,\tau_0)$ is regular.
\end{proof}

\begin{theorem}\label{theorem-2.5}
A topological group $G$ is $H$-closed in the class of semitopological inverse semigroups with continuous inversion if and only if $G$ is compact.
\end{theorem}

\begin{proof}
The implication $(\Leftarrow)$ is trivial.

$(\Rightarrow)$ Let a topological group $G$ be $H$-closed in the class of semitopological inverse semigroups with continuous inversion. Suppose to the contrary: the space $G$ is not compact. Then $G$ is $H$-closed in the class of topological groups and hence it is Ra{\v{\i}}kov complete. If $G$ is precompact then by Theorem~3.7.15 of \cite{ArhangelskiiTkachenko2008}, $G$ is compact. Hence the topological group $G$ is not precompact. This contradicts Proposition~\ref{proposition-2.4}. The obtained contradiction implies the statement of our theorem.
\end{proof}

Theorem~\ref{theorem-2.5} implies the following two corollaries:

\begin{corollary}\label{corollary-2.6}
A topological group $G$ is absolutely $H$-closed in the class of semitopological inverse semigroups with continuous inversion if and only if $G$ is compact.
\end{corollary}

\begin{corollary}\label{corollary-2.7}
A topological group $G$ is $H$-closed in the class of semitopological semigroups if and only if $G$ is compact.
\end{corollary}

The following example shows that there exists a non-compact quasitopological group with adjoined zero which $H$-closed in the class of semitopological inverse semigroups with continuous inversion.

\begin{example}\label{example-2.8}
Let $\mathbb{R}$ be the additive group of real numbers with usual topology. We put $G$ is the direct quare of $\mathbb{R}$ with the product topology. It is well known that $G$ is a topological group. Let $G^0$ be the group $G$ with the adjoined zero $0$. We define the topology $\tau$ on $G^0$ in the following way. For every non-zero element $x$ of $G^0$ the base of the topology $\tau$ at $x$ coincides with base of the product topology at $x$ in $G$. For every $(x_0,y_0)\in\mathbb{R}^2$ and every $\varepsilon>0$ we denote by
\begin{equation*}
    O_{\varepsilon}(x_0,y_0)=\left\{(x,y)\in\mathbb{R}^2\colon \sqrt{(x-x_0)^2+(y-y_0)^2}\leqslant\varepsilon\right\}
\end{equation*}
the usual closed $\varepsilon$-ball with the center at the point  $(x_0,y_0)$. We denote
\begin{equation*}
    A(x_0,y_0)=\left\{(x_0,y)\in\mathbb{R}^2\colon y\in\mathbb{R}\right\}\cup \left\{(x,y_0)\in\mathbb{R}^2\colon x\in\mathbb{R}\right\}
\end{equation*}
and
\begin{equation*}
    U_{\varepsilon}(x_0,y_0)=G^0\setminus\left(O_{\varepsilon}(x_0,y_0)\cup A(x_0,y_0)\right).
\end{equation*}
Now we put $\mathscr{P}(0)=\left\{U_{\varepsilon}(x,y)\colon (x,y)\in\mathbb{R}^2, \varepsilon>0 \right\}$ and $\mathscr{P}(0)\cup\mathscr{B}_G$ is a subbase of the topology $\tau$ on $G^0$, where $\mathscr{B}_G$ is a base of the topology of the topological group $G$. Simple verifications show that $(G^0,\tau)$ is a Hausdorff semitopological inverse semigroup with continuous inversion and $(G^0,\tau)$ is not a regular space.

Then for any finitely many points $(x_1,y_1),\ldots,(x_n,y_n) \in\mathbb{R}^2$ and finitely many $\varepsilon_1,\ldots,\varepsilon_n>0$ the following conditions hold:
\begin{itemize}
  \item[$(a)$] $O_{\varepsilon_1}(x_1,y_1)\cup\cdots\cup O_{\varepsilon_n}(x_n,y_n)$ is a compact subset of the space $(G^0,\tau)$;
  \item[$(b)$] $\operatorname{cl}_{G^0}\left(U_{\varepsilon_1}(x_1,y_1)\cap\cdots \cap U_{\varepsilon_n}(x_n,y_n)\right)\cup O_{\varepsilon_1}(x_1,y_1)\cup\cdots\cup O_{\varepsilon_n}(x_n,y_n)=G^0$.
\end{itemize}
This implies that $(G^0,\tau)$ is an $H$-closed topological space and hence the semigroup $(G^0,\tau)$ is $H$-closed in the class of semitopological inverse semigroups with continuous inversion.
\end{example}

\section{On the closure of a semilattice in a semitopological inverse semigroup with continuous inversion}\label{s3}

It is well known that the subset of idempotent $E(S)$ of a topological semigroup $S$ is a closed subset of $S$ (see: \cite[Vol.~1, Theorem~1.5]{CarruthHildebrantKoch}). We observe that for semitopological semigroups this statement does not hold \cite{Ruppert1984}. Amassing, but the subset of all idempotent $E(S)$ of a semitopological inverse semigroup $S$ with continuous inversion is a closed subset of $S$.

\begin{proposition}\label{proposition-3.1}
The subset of idempotents $E(S)$ of a semitopological inverse semigroup $S$ with continuous inversion is a closed subset of $S$.
\end{proposition}

\begin{proof}
First we observe that for any topological space $X$ and any continuous map $f\colon X\to X$ the set $\operatorname{\textsf{Fix}}(f)$ of fixed point of $f$ is closed subset of $X$ (see: \cite[Vol.~1, Theorem~1.4]{CarruthHildebrantKoch} or \cite[Theorem~1.5.4]{Engelking1989}). Since $e^{-1}=e$ for every idempotent $e\in S$, the continuity of inversion implies that $E(S)\subseteq\operatorname{\textsf{Fix}}(\operatorname{\textbf{inv}})$. Let be $x\in S$ such that $x\in\operatorname{\textsf{Fix}}(\operatorname{\textbf{inv}})$. Since $S$ is an inverse semigroup we obtain that $xx=xx^{-1}\in E(S)$ and hence $\operatorname{\textsf{Fix}}(\operatorname{\textbf{inv}})\subseteq E(S)$. This completes the proof of the proposition.
\end{proof}

Proposition~\ref{proposition-3.1} implies the following

\begin{corollary}\label{corollary-3.2}
The closure of a subsemilattice in a semitopological inverse semigroup $S$ with continuous inversion is a subsemilattice of $S$.
\end{corollary}

Since the closure of a subsemilattice in a Hausdorff topological semigroup is again a topological semilattice, an (absolutely) $H$-closed topological semilattice in the class of topological semilattices is (absolutely) $H$-closed in the class of topological semigroups \cite{GutikPavlyk2003}. In \cite{Stepp1975} Stepp proved that an algebraic semilattice $E$ is algebraically $h$-complete in the class of topological semilattices if and only if every chain in $E$ is finite. The following example shows that for every infinite cardinal $\lambda$ there exists an algebraically $h$-complete semilattice $E(\lambda)$ in the class of topological semilattices of cardinality $\lambda$ such that $E(\lambda)$ with the discrete topology is not $H$-closed in the class of semitopological semigroups.

\begin{example}\label{example-3.3}
Let $\lambda$ be any infinite cardinal. We fix an arbitrary $a_0\in\lambda$ and define the semigroup operation on $\lambda$ by the formula:
\begin{equation*}
    xy=
\left\{
  \begin{array}{ll}
    x, & \hbox{if~} x=y;\\
    a_0, & \hbox{if~} x\neq y.
  \end{array}
\right.
\end{equation*}
The cardinal $\lambda$ with so defined semigroup operation we denote by $E(\lambda)$. It is obvious that $E(\lambda)$ is a semilattice such that $a_0$ is zero of $E(\lambda)$ and any two distinct non-zero elements of $E(\lambda)$ are incomparable with respect to the natural partial order on $E(\lambda)$. Let be $a\notin E(\lambda)$. We extend the semigroup operation from $E(\lambda)$ onto $S=E(\lambda)\cup\{a\}$ in the following way:
\begin{equation*}
    aa=ax=xa=a_0, \qquad \mbox{for any~} \; x\in E(\lambda).
\end{equation*}
It is obvious that $S$ with so defined operation is not a semilattice.

We define a topology $\tau$ on $S$ in the following way. Fix an arbitrary sequence of distinct points $\{x_n\colon n\in\mathbb{N}\}$ from $E(\lambda)$ and put $U_n(a)=\{a\}\cup\{x_i\colon i\geqslant n\}$. Put all elements of the set $E(\lambda)$ are isolated points of the space $(S,\tau)$ and the family $\mathscr{B}(a)=\{U_n(a)\colon n\in\mathbb{N}\}$ is a base of the topology $\tau$ at the point $a\in S$. Simple verifications show that $(S,\tau)$ is a metrizable $0$-dimensional semitopological semigroup and $E(\lambda)$ is a dense subsemilattice of $(S,\tau)$. Also, we observe that by Theorem~9 from \cite{Stepp1975} the semilattice $E(\lambda)$ is algebraically $h$-complete in the class of topological semilattices.
\end{example}

\begin{remark}
We observe that for every infinite cardinal $\lambda$ and every Hausdorff topology $\tau$ on $E(\lambda)$ such that $(E(\lambda),\tau)$ is a semitopological semilattice we have that all non-zero idempotents of $(E(\lambda),\tau)$ are isolated points and moreover $(E(\lambda),\tau)$ is a topological semilattice. Also, a simple modification of the proof in the Example~\ref{example-3.3} shows that a semitopological semilattice $(E(\lambda),\tau)$ is $H$-closed in the class of semitopological semigroups if and only if the space $(E(\lambda),\tau)$ is compact.
\end{remark}

Suppose that $E$ is a Hausdorff semitopological semilattice. If $L$ is a maximal chain in $E$, then by Proposition~IV-1.13 of \cite{GierzHofmannKeimelLawsonMisloveScott2003} we have that $L=\bigcap_{e\in L}({\uparrow}e\cup{\downarrow}e)$ is a closed subset of $E$ and hence we proved the following proposition:

\begin{proposition}\label{proposition-3.4}
The closure of a linearly ordered subsemilattice of a Hausdorff semitopological semilattice $E$ is a linearly ordered subsemilattice of $E$.
\end{proposition}

It is well known that the natural partial order on a Hausdorff semitopological semilattice is semiclosed (see \cite[Proposition~IV-1.13]{GierzHofmannKeimelLawsonMisloveScott2003}). Also, by Lemma~3 of \cite{Ward1954} a semiclosed linear order is closed, and hence every linearly ordered set with a closed order admits the structure of a Hausdorff topological semilattice. This implies the following proposition:

\begin{proposition}\label{proposition-3.5}
Every linearly ordered Hausdorff semitopological semilattice is a topological semilattice.
\end{proposition}

Propositions~\ref{proposition-3.4} and~\ref{proposition-3.5} imply

\begin{theorem}\label{theorem-3.6}
A Hausdorff linearly ordered topological semilattice $E$ is $H$-closed in the class of semitopological semilattices if and only if $E$ is $H$-closed in the class of topological semilattices.
\end{theorem}

Theorem~\ref{theorem-3.6} and results obtained in the paper \cite{GutikRepovs2008} imply Corollaries~\ref{corollary-3.7}---\ref{corollary-3.11}.

A linearly ordered semilattice $E$ is called \emph{complete} if every non-empty subset of $S$ has $\inf$ and $\sup$.

\begin{corollary}\label{corollary-3.7}
A linearly ordered semitopological semilattice $E$ is $H$-closed in the class of semitopological semilattices if and only if the following conditions hold:
\begin{itemize}
    \item[$(i)$] $E$ is complete;
    \item[$(ii)$] $x=\sup A$ for $A={\downarrow}A\setminus\{ x\}$
    implies $x\in\operatorname{cl}_EA$, whenever
    $A\neq\varnothing$; and
    \item[$(iii)$] $x=\inf B$ for $B={\uparrow}B\setminus\{ x\}$
    implies $x\in\operatorname{cl}_EB$, whenever
    $B\neq\varnothing$
\end{itemize}
\end{corollary}

\begin{corollary}\label{corollary-3.8}
Every linearly ordered $H$-closed semitopological semilattice in the class of semitopological semilattices is absolutely $H$-closed in the class of semitopological semilattices.
\end{corollary}

\begin{corollary}\label{corollary-3.9}
Every linearly ordered $H$-closed semitopological semilattice in the class of semitopological semilattices contains maximal and minimal idempotents.
\end{corollary}

\begin{corollary}\label{corollary-3.10}
Let $E$ be a linearly ordered $H$-closed semitopological semilattice in the class of semitopological semilattices and $e\in E$. Then ${\uparrow}e$ and ${\downarrow}e$ are (absolutely) $H$-closed topological semilattices in the class of semitopological semilattices.
\end{corollary}

\begin{corollary}\label{corollary-3.11}
Every linearly ordered semitopological semilattice is a dense subsemilattice of an $H$-closed semitopological semilattice in the class of semitopological semilattices. \end{corollary}

\begin{remark}\label{remark-3.12}
Theorem~\ref{theorem-3.6}, Example~7 and Proposition~8 from \cite{GutikRepovs2008} imply that there exists a countable linearly ordered $\sigma$-compact $0$-dimensional scattered locally compact metrizable topological semilattice which does not embeds into any compact Hausdorff semitopological semilattice.
\end{remark}

At the finish of this section we construct an $H$-closed semitopological semilattice in the class of semitopological semilattices which is not absolutely $H$-closed in the class of semitopological semilattices.

A filter $\mathscr{F}$ on a set $X$ is called \emph{free} if $\bigcap\mathscr{F}=\varnothing$.

\begin{example}[\cite{BardylaGutik2012}]\label{example-3.13}
Let $\mathbb{N}$ denote the set of positive integers. For
each free filter $\mathscr{F}$ on $\mathbb{N}$ consider the topological space $\mathbb{N}_{\mathscr{F}}=\mathbb{N}\cup \{\mathscr{F}\}$ in which all points $x\in\mathbb{N}$ are isolated while the sets $F\cup\{\mathscr{F}\}$, $F\in\mathscr{F}$, form a neighbourhood base at the unique non-isolated point $\mathscr{F}$.

The semilattice operation $\min$ of $\mathbb{N}$ extends to a continuous semilattice operation $\min$ on $\mathbb{N}_{\mathscr{F}}$ such that $\min\{n,\mathscr{F}\}= \min\{\mathscr{F},n\}=n$  and $\min\{\mathscr{F},\mathscr{F}\}=\mathscr{F}$ for all $n\in\mathbb{N}$. By $\mathbb{N}_{\mathscr{F},\min}$  we shall denote the topological space $\mathbb{N}_{\mathscr{F}}$ with the semilattice operation $\min$. Simple verifications show that $\mathbb{N}_{\mathscr{F},\min}$ is a topological semilattice. Then by Theorem~2$(i)$ of \cite{BardylaGutik2012} the topological semilattice $\mathbb{N}_{\mathscr{F},\min}$ is $H$-closed in the class of topological semilattices and hence by Theorem~\ref{theorem-3.6} it is $H$-closed in the class of semitopological semilattices.
\end{example}

Later by $E_2=\{0,1\}$ we denote the discrete topological semilattice with the semilattice operation $\min$.

\begin{theorem}\label{theorem-3.14}
Let $\mathscr{F}$ be a free filter on $\mathbb{N}$ and $F\in\mathscr{F}$ be a set with infinite complement $\mathbb{N}\setminus F$. Then the closed subsemilattice $E=\left(\mathbb{N}_{\mathscr{F},\min}\times\{0\}\right)\cup \left((\mathbb{N}\setminus F)\times\{1\}\right)$ of the direct product $\mathbb{N}_{\mathscr{F},\min}\times E_2$ is $H$-closed not absolutely $H$-closed in the class of semitopological semilattices.
\end{theorem}

\begin{proof}
The definition of the topological semilattice $\mathbb{N}_{\mathscr{F},\min}\times E_2$ implies that $E$ is a closed subsemilattice of $\mathbb{N}_{\mathscr{F},\min}\times E_2$.

Suppose the contrary: the topological semilattice $E$ is not $H$-closed in the class of semitopological semilattices. Since the closure of a subsemilattice in a semitopological semilattice is a semilattice (see \cite[Chapter~I,~Proposition~1.8$(ii)$]{Ruppert1984}) we conclude that there exists a semitopological semilattice $S$ which contains $E$ as a dense subsemilattice and $S\setminus E\neq\varnothing$. We fix an arbitrary $a\in S\setminus E$. Then for every open neighbourhood $U(a)$ of the point $a$ in $S$ we have that the set $U(a)\cap E$ is infinite. By Theorem~2$(i)$ of \cite{BardylaGutik2012} and Theorem~\ref{theorem-3.6}, the subspace $\mathbb{N}_{\mathscr{F},\min}\times\{0\}$ of $E$ with the induced semilattice operation from $E$ is an $H$-closed in the class of semitopological semilattices. Therefore there exists an open neighbourhood $U(a)$ of the point $a$ in $S$ such that $U(a)\cap E\subseteq (\mathbb{N}\setminus F)\times\{1\}$ and hence the set $U(a)\cap((\mathbb{N}\setminus F)\times\{1\})$ is infinite.

Since the subset $\mathbb{N}_{\mathscr{F},\min}\times\{0\}$ is an ideal of $E$, the $H$-closedness of $\mathbb{N}_{\mathscr{F},\min}\times\{0\}$ in the class of semitopological semilattices implies that $\mathbb{N}_{\mathscr{F},\min}\times\{0\}$ is a closed ideal in $S$ and hence we have that $x\cdot a\in \mathbb{N}_{\mathscr{F},\min}\times\{0\}$ for every $x\in \mathbb{N}_{\mathscr{F},\min}\times\{0\}$. Since for every open neighbourhood $U(a)$ of the point $a$ in $S$ the set $U(a)\cap((\mathbb{N}\setminus F)\times\{1\})$ is infinite the semilattice operation in $E$ implies that for every $x\in (\mathbb{N}_{\mathscr{F},\min}\times\{0\})\setminus\{(\mathscr{F},0)\}$ the set $x\cdot U(a)$ is infinite and hence we have that $x\cdot a\notin N\times\{0\}=(\mathbb{N}_{\mathscr{F},\min}\times\{0\})\setminus\{(\mathscr{F},0)\}$. Therefore we obtain that $x\cdot a=(\mathscr{F},0)$. Now, since in $\mathbb{N}_{\mathscr{F},\min}$ the sets $F\cup\{\mathscr{F}\}$, $F\in\mathscr{F}$, form a neighbourhood base at the unique non-isolated point $\mathscr{F}$, we conclude that $x\cdot U(a)\nsubseteq (F\cup\{\mathscr{F}\})\times \{0\}$, which contradicts the separate continuity of the semilattice operation on $S$. Hence we get that $S\setminus E=\varnothing$. This implies that the topological semilattice $E$ is $H$-closed in the class of semitopological semilattices.

Now, by Theorem~3 of \cite{BardylaGutik2012} the topological semilattice $E$ is not absolutely $H$-closed in the class of topological semilattices, and hence $E$ is not absolutely $H$-closed in the class of semitopological semilattices.
\end{proof}

\begin{remark}\label{remark-3.15}
Corollary~\ref{corollary-3.2} implies that the topological semilattice $E$ determined in Theorem~\ref{theorem-3.14} is an example a topological inverse semigroup which is  $H$-closed but is not absolutely $H$-closed in the class of semitopological semigroups with continuous inversion.
\end{remark}

\begin{remark}\label{remark-3.16}
Proposition~\ref{proposition-3.5} and Theorem~\ref{theorem-3.6} imply that Theorem~2 of \cite{BardylaGutik2012} describes all $H$-closed semilattices in the class of semitopological semilattices which contain the discrete semilattice $(\mathbb{N}, \min)$ or the discrete semilattice $(\mathbb{N},\max)$ as a dense subsemilattice.
\end{remark}

\section{On the closure of topological Brandt $\lambda$-extensions in a semitopological inverse semigroup with continuous inversion}\label{s4}

In this section we study the preserving of $H$-closedness and absolute $H$-closedness by topological Brandt $\lambda^0$-extensions and orthogonal sums of semitopological semigroups.

\begin{theorem}\label{theorem-4.1}
Let $S$ be a Hausdorff semitopological inverse monoid with zero and continuous inversion. Then the following conditions are equivalent:
\begin{enumerate}
  \item[$(i)$] $S$ is absolutely $H$-closed in the class of semitopological inverse semigroups with continuous inversion;
  \item[$(ii)$] there exists a cardinal $\lambda\geqslant 2$ such that every topological Brandt $\lambda^0$-extension of $S$ is absolutely $H$-closed in the class of semitopological inverse semigroups with continuous inversion;
  \item[$(iii)$] for each cardinal $\lambda\geqslant 2$ every topological Brandt $\lambda^0$-extension of $S$ is absolutely $H$-closed in the class of semitopological inverse semigroups with continuous inversion.
\end{enumerate}
\end{theorem}

\begin{proof}
$(i)\Rightarrow (iii)$. Suppose that the semigroup $S$ is absolutely $H$-closed in the class of semitopological inverse semigroups with continuous inversion. We fix an arbitrary cardinal $\lambda\geqslant 2$. Let $B_\lambda^0(S)$ be a topological Brandt $\lambda^0$-extension of $S$ in the class of semitopological inverse semigroups with continuous inversion, $T$ be a semitopological inverse semigroup with continuous inversion  and $h\colon B_\lambda^0(S)\to T$ be a continuous homomorphism.

First we observe that by Proposition~2.3 of \cite{GutikRepovs2010}, either $h$ is an annihilating homomorphism or the image $(B_{\lambda}^0(S))h$ is isomorphic to the Brandt $\lambda^0$-extension $B_{\lambda}^0((S_{\alpha,\alpha})h)$ of the semigroup $(S_{\alpha,\alpha})h$ for some $\alpha\in\lambda$. If $h$ is an annihilating homomorphism then $(S_{\alpha,\alpha})h$ is a singleton, and therefore we have that $(S_{\alpha,\alpha})h$ is a closed subset of $T$. Hence, later we assume that $h$ is a non-annihilating homomorphism.

Next we show that for any $\gamma,\delta\in\lambda$ the set $(S_{\gamma,\delta})h$ is closed in the space $T$. By Definition~\ref{def2} there exists $\alpha\in \lambda$ such that $(S_{\alpha,\alpha})h$ is a closed subset of $T$. We define the maps $\varphi_h,\psi_h\colon T\to T$ by the formulae $(x)\varphi_h=(\alpha,1_S,\gamma)h\cdot (x)h\cdot(\delta,1_S,\alpha)h$ and $(x)\psi_h=(\gamma,1_S,\alpha)h\cdot (x)h\cdot(\alpha,1_S,\delta)h$. Then the maps $\varphi_h$ and $\psi_h$ are continuous because left and right translations in $T$ and homomorphism $h\colon B_\lambda^0(S)\to T$ are continuous maps. Thus, the full preimage $A=((S_{\alpha,\alpha})h)\varphi^{-1}_h$ is a closed subset of $T$. Then the restriction map $(\varphi_h\circ\psi_h)|_{A}\colon A\to (S_{\gamma,\delta})h$ is a retraction, and therefore the set $(S_{\gamma,\delta})h$ is a retract of $A$. This implies that $(S_{\gamma,\delta})h$ is a closed subset of $T$.

Suppose to the contrary that $(B_\lambda^0(S))h$ is not a closed subsemigroup of $T$. By Lemma~II.1.10 of \cite{Petrich1984}, $(B_\lambda^0(S))h$ is an inverse subsemigroup of $T$. Since by Proposition~\ref{proposition-2.2} the closure of an inverse subsemigroup $(B_\lambda^0(S))h$ in a semitopological inverse semigroup $T$ with continuous inversion is an inverse semigroup, without loss of generality we may assume that $(B_\lambda^0(S))h$ is a dense proper inverse subsemigroup of $T$.

We fix an arbitrary $x\in \operatorname{cl}_{T}((B_\lambda^0(S))h)\setminus (B_\lambda^0(S))h$. Then only one of the following cases holds:
\begin{enumerate}
  \item[$a)$] $x$ is an idempotent of the semigroup $T$;
  \item[$b)$] $x$ is a non-idempotent element of $T$.
\end{enumerate}

Suppose that case $a)$ holds. By the previous part of the proof we have that every open neighbourhood $U(x)$ of the point $x$ in the topological space $T$ intersects infinitely sets of the form $(S_{\alpha,\beta})h$, $\alpha,\beta\in\lambda$. By Proposition~2.3 of \cite{GutikRepovs2010}, $(B_{\lambda}^0(S))h$ is isomorphic to the Brandt $\lambda^0$-extension $B_{\lambda}^0((S_{\alpha,\alpha})h)$ of the semigroup $(S_{\alpha,\alpha})h$ for some $\alpha\in\lambda$, and since $(B_{\lambda}^0(S))h$ is a dense subsemigroup of semitopological semigroup $T$, the zero $0$ of the semigroup $(B_{\lambda}^0(S))h$ is zero of $T$ (see \cite[Lemma~23]{GutikPavlyk2013a}). Then the semigroup operation of $B_{\lambda}^0((S_{\alpha,\alpha})h)$ implies that either $0\in (\alpha,e,\alpha)h\cdot U(x)$ or $0\in U(x)\cdot(\alpha,e,\alpha)h$ for every non-zero idempotent $(\alpha,e,\alpha)$ of $B_{\lambda}^0(S)$, $e\in E(S)$, $\alpha\in\lambda$. Now by the Hausdorffness of the space $T$ and the separate continuity of the semigroup operation of $T$ we have that either $(\alpha,e,\alpha)h\cdot x=0$ or $x\cdot(\alpha,e,\alpha)h=0$ for every non-zero idempotent $(\alpha,e,\alpha)$ of $B_{\lambda}^0(S)$, $e\in E(S)$, $\alpha\in\lambda$. Since in an inverse semigroup any two idempotents commute we conclude that $(\alpha,e,\alpha)h\cdot x=x\cdot(\alpha,e,\alpha)h=0$ for every non-zero idempotent $(\alpha,e,\alpha)$ of the semigroup $B_{\lambda}^0(S)$.

We fix an arbitrary non-zero element $(\alpha,s,\beta)h$ of the semigroup $(B_{\lambda}^0(S))h$, where $\alpha,\beta\in\lambda$ and $s\in S^*$. Then by the previous part of the proof we obtain that
\begin{equation*}
    x\cdot (\alpha,s,\beta)h=x\cdot (\alpha,ss^{-1}s,\beta)h= x\cdot ((\alpha,ss^{-1},\alpha)(\alpha,s,\beta))h= x\cdot (\alpha,ss^{-1},\alpha)h\cdot(\alpha,s,\beta)h=0\cdot(\alpha,s,\beta)h=0
\end{equation*}
and
\begin{equation*}
    (\alpha,s,\beta)h\cdot x=(\alpha,ss^{-1}s,\beta)h\cdot x= ((\alpha,s,\beta)(\beta,s^{-1}s,\beta))h\cdot x= (\alpha,s,\beta)h\cdot (\beta,s^{-1}s,\beta)h\cdot x= (\alpha,s,\beta)h\cdot 0=0.
\end{equation*}
This implies that for every open neighbourhood $U(x)$ of the point $x$ in the space $T$ we have that $0\in x\cdot U(x)$ and $0\in U(x)\cdot x$. Then by Hausdorffness of the space $T$ and the separate continuity of the semigroup operation in $T$ we get that $x\cdot x=0$, and hence $x=0$. This implies that $E(T)=E((B_{\lambda}^0(S))h)$.

Suppose that case $b)$ holds. If $xx^{-1}=0$, then $x=xx^{-1}x=0\cdot x=0$, and similarly if $x^{-1}x=0$, then $x=xx^{-1}x=x\cdot 0=0$. This implies that $xx^{-1}, x^{-1}x\in E((B_{\lambda}^0(S))h)\setminus\{0\}$.

Then by Lemma~I.7.10 of \cite{Petrich1984} there exist idempotents $(\alpha,e,\alpha),(\beta,f,\beta)\in B_{\lambda}^0(S)$ such that $xx^{-1}=(\alpha,e,\alpha)h$ and $x^{-1}x=(\beta,f,\beta)h$, where $e,f\in (E(S))^*$ and $\alpha,\beta\in\lambda$. Then we have that $x\cdot(\beta,f,\beta)h=(\alpha,e,\alpha)h\cdot x=x$. Since $x\in \operatorname{cl}_T((B_{\lambda}^0(S))h)\setminus (B_{\lambda}^0(S))h$, every open neighbourhood $U(x)$ of the point $x$ in the space $T$ intersects infinitely many sets $(S_{\gamma,\delta})h$, $\gamma,\delta\in\lambda$, and hence we obtain that either $U(x)\cdot(\beta,f,\beta)h\ni 0$ or $(\alpha,e,\alpha)h\cdot U(x)\ni 0$. Then the Hausdorffness of the space $T$ and the separate continuity of the semigroup operation on $T$ imply that $x\cdot(\beta,f,\beta)h=0$ or $(\alpha,e,\alpha)h\cdot x=0$. If $x\cdot(\beta,f,\beta)h=0$ then $x=x\cdot xx^{-1}=x\cdot(\beta,f,\beta)h=0$ and if $(\alpha,e,\alpha)h\cdot x=0$ then $x=xx^{-1}x=(\alpha,e,\alpha)h\cdot x=0$. All these two cases imply that $x=0$, and hence we get that $T=(B_{\lambda}^0(S))h$, which completes the proof of our theorem.

The implication $(iii)\Rightarrow (ii)$ is trivial.

$(ii)\Rightarrow (i)$. Suppose to the contrary: there exists semigroup $S$ such that $S$ is not absolutely $H$-closed semigroup $S$ in the class of semitopological inverse semigroups with continuous inversion and condition $(ii)$ holds for $S$. Then there exists a semitopological inverse semigroup $T$ with continuous inversion and continuous homomorphism $h\colon S\to T$ such that $(S)h$ is non-closed subset of $T$. Now, by Proposition~\ref{proposition-2.2} without loss of generality we may assume that $(S)h$ is a proper dense inverse subsemigroup of $T$.

Next, for the cardinal $\lambda$ we define topologies $\tau_T^B$ and $\tau_S^B$ on Brandt $\lambda^0$-extensions $B_{\lambda}^0(T)$ and $B_{\lambda}^0(S)$, respectively, in the following way. We put
\begin{equation*}
    \mathscr{B}_{(\alpha,t,\beta)}^T=\left\{(U(t))_{\alpha,\beta}\colon 0\notin U(t)\in\mathscr{B}_T(t)\right\} \qquad \mbox{and} \qquad \mathscr{B}_{(\alpha,s,\beta)}^s=\left\{(U(s))_{\alpha,\beta}\colon 0\notin U(s)\in\mathscr{B}_S(s)\right\}
\end{equation*}
are bases of topologies $\tau_T^B$ and $\tau_S^B$ at non-zero elements $(\alpha,t,\beta)\in B_{\lambda}^0(T)$ and $(\alpha,s,\beta)\in B_{\lambda}^0(S)$, respectively, $\alpha,\beta\in\lambda$, where $\mathscr{B}_T(t)$ and $\mathscr{B}_S(s)$ are bases of topologies of spaces $T$ and $S$ at non-zero elements $t\in T$ and $s\in S$, respectively. Also, if $\mathscr{B}_T(0_T)$ and $\mathscr{B}_S(0_S)$ are bases at zeros $0_T\in T$ and $0_S\in S$ then we define
\begin{equation*}
    \mathscr{B}_{0}^T=\Big\{\{0\}\cup \bigcup_{\alpha,\beta\in\lambda}\!(U(0_T))_{\alpha,\beta}^*\colon  U(0_T)\in\mathscr{B}_T(0_T)\Big\} \quad \mbox{and} \quad \mathscr{B}_{0}^S=\Big\{\{0\}\cup \bigcup_{\alpha,\beta\in\lambda}\!(U(0_S))_{\alpha,\beta}^*\colon  U(0_S)\in\mathscr{B}_S(0_S)\Big\}
\end{equation*}
to be the bases of topologies $\tau_T^B$ and $\tau_S^B$ at zeros $0\in B_{\lambda}^0(T)$ and $0\in B_{\lambda}^0(S)$, respectively.

Simple verifications show that if $T$ and $S$ are semitopological inverse semigroups with continuous inversion, then so are $(B_{\lambda}^0(T),\tau_T^B)$ and $(B_{\lambda}^0(S),\tau_S^B)$. Also the continuity of homomorphism $h\colon S\to T$ implies that the map $h_B\colon B_{\lambda}^0(S)\to B_{\lambda}^0(T)$ defined by the formulae
\begin{equation*}
(\alpha,s,\beta)h_B=
\left\{
  \begin{array}{cl}
    (\alpha,(s)h,\beta), & \hbox{if~} (s)h\neq 0_T;\\
    0, & \hbox{otherwise},
  \end{array}
\right.
\end{equation*}
$s\in S^*$, $\alpha,\beta\in\lambda$, and $(0)h_B=0$ is continuous. Also, by Theorem~3.10 of \cite{GutikRepovs2010} so defines map $h_B\colon B_{\lambda}^0(S)\to B_{\lambda}^0(T)$ is a homomorphism. The definition of the topology $\tau_T^B$ on $B_{\lambda}^0(T)$ implies that the homomorphic image $(B_{\lambda}^0(S))h_B$ is a dense proper subsemigroup of the semitopological inverse semigroup $(B_{\lambda}^0(T),\tau_T^B)$ with continuous inversion, which contradicts to statement $(ii)$. The obtained contradiction implies the requested implication.
\end{proof}

Now, if we put $h$ is a topological isomorphic embedding of semitopological semigroups with continuous inversions in the proof of Theorem~\ref{theorem-4.1}, then we get the proof of the following theorem:

\begin{theorem}\label{theorem-4.2}
Let $S$ be a Hausdorff semitopological inverse monoid with zero and continuous inversion. Then the following conditions are equivalent:
\begin{enumerate}
  \item[$(i)$] $S$ is $H$-closed in the class of semitopological inverse semigroups with continuous inversion;
  \item[$(ii)$] there exists a cardinal $\lambda\geqslant 2$ such that every topological Brandt $\lambda^0$-extension of $S$ is $H$-closed in the class of semitopological inverse semigroups with continuous inversion;
  \item[$(iii)$] for each cardinal $\lambda\geqslant 2$ every topological Brandt $\lambda^0$-extension of $S$ is $H$-closed in the class of semitopological inverse semigroups with continuous inversion.
\end{enumerate}
\end{theorem}

Theorem~\ref{theorem-4.1} implies Corollary~\ref{corollary-4.3} which generalizes Corollary~20 from \cite{GutikPavlyk2005}.

\begin{corollary}\label{corollary-4.3}
For any cardinal $\lambda\geqslant 2$ the semigroup of $\lambda\times\lambda$-units $B_\lambda$ is algebraically h-complete in the class of semitopological inverse semigroups with continuous inversion.
\end{corollary}

Also, Theorems~\ref{theorem-4.1} and~\ref{theorem-4.2} imply the following corollary:

\begin{corollary}\label{corollary-4.4}
For an inverse monoid $S$ with zero the following conditions are equivalent:
\begin{enumerate}
  \item[$(i)$] $S$ is algebraically complete $($algebraically $h$-complete$)$ in the class of semitopological inverse semigroups with continuous inversion;
  \item[$(ii)$] there exists a cardinal $\lambda\geqslant 2$ such the Brandt $\lambda^0$-extension of $S$ is algebraically complete $($algebraically $h$-complete$)$ in the class of semitopological inverse semigroups with continuous inversion;
  \item[$(iii)$] for each cardinal $\lambda\geqslant 2$ the Brandt $\lambda^0$-extension of $S$ is algebraically complete $($algebraically $h$-complete$)$ in the class of semitopological inverse semigroups with continuous inversion.
\end{enumerate}
\end{corollary}

Theorems~\ref{theorem-4.3}, \ref{theorem-4.4} and~\ref{theorem-4.5} give a method of the construction of absolutely $H$-closed and $H$-closed semigroups in the class of semitopological inverse semigroups with continuous inversion.

\begin{theorem}\label{theorem-4.3}
Let $S=\bigcup_{\alpha\in\mathscr{A}}S_{\alpha}$ be a semitopological inverse semigroup with continuous inversion such that
\begin{itemize}
    \item[$(i)$] $S_{\alpha}$ is an absolutely $H$-closed semigroup in the class of semitopological inverse semigroups with continuous inversion for any $\alpha\in\mathscr{A}$; and
    \item[$(ii)$] there exists an ideal $T$ of $S$ which is absolutely $H$-closed in the class of semitopological inverse semigroups with continuous inversion such that $S_{\alpha}\cdot S_{\beta}\subseteq T$ for all $\alpha\neq\beta$, $\alpha,\beta\in\mathscr{A}$.
\end{itemize}
Then $S$ is an absolutely $H$-closed semigroup in the class of semitopological inverse semigroups with continuous inversion.
\end{theorem}

\begin{proof}
Suppose to the contrary that there exists a semitopological inverse semigroup $K$ with continuous inversion and continuous homomorphism $h\colon S\to K$ such that the image $(S)h$ is not a closed subsemigroup of $K$. By Lemma~II.1.10 of \cite{Petrich1984}, $(S)h$ is an inverse subsemigroup of $K$. Since by Proposition~\ref{proposition-2.2} the closure $\operatorname{cl}_K((S)h)$ of an inverse subsemigroup $(S)h$ in a semitopological inverse semigroup $K$ with continuous inversion is an inverse semigroup, without loss of generality we may assume that $(S)h$ is a dense proper inverse subsemigroup of $K$.

We observe that the assumption of the theorem states that $T$ is an ideal of $S$. This implies that $(T)h$ is an ideal in $(S)h$. Then by Proposition~I.1.8$(iii)$ of \cite{Ruppert1984} the closure of an ideal of a semitopological semigroup is again an ideal, and hence we get that $(T)h$ is a closed ideal of the semigroup $K$.

We fix an arbitrary $x\in K\setminus (S)h$. Then only one of the following cases holds:
\begin{enumerate}
  \item[$a)$] $x$ is an idempotent of the semigroup $K$;
  \item[$b)$] $x$ is a non-idempotent element of $K$.
\end{enumerate}

First we show that $x\cdot y, y\cdot x\in (T)h$ for every $y\in (S)h$. We fix an arbitrary open neighbourhood $U(x)$ of the point $x$ in the space $K$. Since $U(x)$ intersects infinitely many subsemigroups of $K$ from the family $\{(S_\alpha)h \colon \alpha\in\mathscr{A} \}$ we conclude that $U(x)\cdot y\cap (T)h\neq\varnothing$ and $y\cdot U(x)\cap (T)h\neq\varnothing$ for every $y\in (S)h$. Then the separate continuity of the semigroup operation in $K$ implies that any open neighbourhoods $W(x\cdot y)$ and  $W(y\cdot x)$ of the points $x\cdot y$ and $y\cdot x$ in $K$, respectively, intersect the ideal $(T)h$. This implies that $x\cdot y, y\cdot x\in \operatorname{cl}_K((T)h)$. Since the ideal $(T)h$ is closed in $K$ we conclude that $x\cdot y, y\cdot x\in (T)h$.

Suppose that case $a)$ holds. Then there exists an open neighbourhood $U(x)$ of the point $x$ in the space $K$ such that $U(x)\cap (T)h=\varnothing$ and the neighbourhood $U(x)$ intersects infinitely many semigroups from the family $\{(S_\alpha)h \colon \alpha\in\mathscr{A} \}$. By the separate continuity of the semigroup operation in $K$ we have that for every open neighbourhood $U(x)$ of the point $x$ in $K$ such that $U(x)\cap(T)h=\varnothing$ there exists an open neighbourhood $V(x)$ of $x$ in $K$ such that $x\cdot V(x)\subseteq U(x)$ and $V(x)\cdot x\subseteq U(x)$. Now, the previous part of proof implies that $x\cdot V(x)\cap (T)h\neq\varnothing$ and $V(x)\cdot x\cap (T)h\neq\varnothing$, which contradict the assumption $U(x)\cap(T)h=\varnothing$. The obtained contradiction implies that $E((S)h)=E(K)$.

Suppose that case $b)$ holds. Then there exist idempotents $e$ and $f$ in $(S)h$ such that $xx^{-1}=e$ and $x^{-1}x=f$. We observe that $e,f\notin (T)h$. Indeed, if $e\in (T)h$ or $f\in (T)h$, then we have that
\begin{equation*}
  x=xx^{-1}x=ex\in (T)h \qquad \mbox{and} \qquad x=xx^{-1}x=xf\in (T)h,
\end{equation*}
because $(T)h$ is an ideal of the semigroup $K$. Since $x\in\operatorname{cl}_K((S)h)$, every open neighbourhood of the point $x$ in $K$ intersects infinitely many semigroups from the family $\{(S_\alpha)h\colon\alpha\in \mathscr{A}\}$, and hence we get that
\begin{equation*}
    (U(x)\cdot f)\cap (T)h\neq\varnothing \qquad \mbox{and} \qquad (e\cdot U(x))\cap (T)h\neq\varnothing.
\end{equation*}
Then the Hausdorffness of $K$ and the separate continuity of the semigroup operation in $K$ imply that $x=xx^{-1}x=x\cdot f=e\cdot x\in (T)h$. This contradicts the assumption that $x\neq (T)h$. The obtained contradiction implies the statement of our theorem.
\end{proof}

The proof of Theorem~\ref{theorem-4.4} is similar to the proof of Theorem~\ref{theorem-4.3}.

\begin{theorem}\label{theorem-4.4}
Let $S=\bigcup_{\alpha\in\mathscr{A}}S_{\alpha}$ be a semitopological inverse semigroup with continuous inversion such that
\begin{itemize}
    \item[$(i)$] $S_{\alpha}$ is an $H$-closed semigroup in the class of semitopological inverse semigroups with continuous inversion for any $\alpha\in\mathscr{A}$; and
    \item[$(ii)$] there exists an ideal $T$ of $S$ which is $H$-closed in the class of semitopological inverse semigroups with continuous inversion such that $S_{\alpha}\cdot S_{\beta}\subseteq T$ for all $\alpha\neq\beta$, $\alpha,\beta\in\mathscr{A}$.
\end{itemize}
Then $S$ is an $H$-closed semigroup in the class of semitopological inverse semigroups with continuous inversion.
\end{theorem}

\begin{theorem}\label{theorem-4.5}
Let a semitopological semigroup $S$ with continuous inversion be the orthogonal sum of a family $\left\{S_{\alpha}\colon \alpha\in\mathscr{I}\right\}$ of semitopological inverse semigroups with zeros. Then $S$ is an $($absolutely$)$ $H$-closed semigroup in the class of semitopological inverse semigroups with continuous inversion if and only if so is any element of the family $\left\{S_{\alpha}\colon \alpha\in\mathscr{I}\right\}$.
\end{theorem}

\begin{proof}
First we observe that if $S$ is a semitopological semigroup with continuous inversion then so is every semigroup from the family $\left\{S_{\alpha}\colon \alpha\in\mathscr{I}\right\}$.

The implication $(\Leftarrow)$ follows from Theorems~\ref{theorem-4.3} and~\ref{theorem-4.4}.

First we shall prove the implication $(\Rightarrow)$ in the case of absolute $H$-closedness.

Suppose to the contrary that there exists an absolute $H$-closed semigroup $S$ in the class of semitopological inverse semigroups with continuous inversion which is an orthogonal sum of a family $\left\{S_{\alpha}\colon \alpha\in\mathscr{I}\right\}$ of semitopological inverse semigroups and there exists a semigroup $S_{\alpha_0}$ in this family such that $S_{\alpha_0}$ is not absolute $H$-closed in the class of semitopological inverse semigroups with continuous inversion. Then there exists a semitopological inverse semigroup $K$ with continuous inversion and continuous homomorphism $h\colon S_{\alpha_0}\to K$ such that the image $(S_{\alpha_0})h$ is not a closed subsemigroup of $K$. By Lemma~II.1.10 of \cite{Petrich1984}, $(S_{\alpha_0})h$ is an inverse subsemigroup of $K$. Since by Proposition~\ref{proposition-2.2} the closure $\operatorname{cl}_K((S_{\alpha_0})h)$ of an inverse subsemigroup $(S_{\alpha_0})h$ in a semitopological inverse semigroup $K$ with continuous inversion is an inverse semigroup, without loss of generality we may assume that $(S_{\alpha_0})h$ is a dense proper inverse subsemigroup of $K$. Also, the semigroup $K$ has zero because $(S_{\alpha_0})h$ contains zero.

We define a map $f\colon S\to K$ by the formula
\begin{equation*}
    (x)f=
\left\{
  \begin{array}{cl}
    0_K, & \hbox{if~} x\in S\setminus S^*_{\alpha_0}; \\
    (x)h, & \hbox{if~} x\in S^*_{\alpha_0},
  \end{array}
\right.
\end{equation*}
where $0_K$ is zero of the semigroup $K$. Simple verifications show that so defined map $f$ is a continuous homomorphism, but the image $(S)f=(S_{\alpha_0})h$ is a dense proper subsemigroup of $K$. This contradicts the assumption that the semigroup $S$ is absolutely $H$-closed semigroup in the class of semitopological inverse semigroups with continuous inversion.

Now, we suppose that there exists an $H$-closed semigroup $S$ in the class of semitopological inverse semigroups with continuous inversion which is an orthogonal sum of a family $\left\{S_{\alpha}\colon \alpha\in\mathscr{I}\right\}$ of semitopological inverse semigroups and there exists a semigroup $S_{\alpha_0}$ in this family such that $S_{\alpha_0}$ is not $H$-closed in the class of semitopological inverse semigroups with continuous inversion. Then there exists a semitopological inverse semigroup $K$ with continuous inversion such that $S_{\alpha_0}$ is not a closed subsemigroup of $K$. Since by Proposition~\ref{proposition-2.2} the closure $\operatorname{cl}_K(S_{\alpha_0})$ of an inverse subsemigroup $S_{\alpha_0}$ in a semitopological inverse semigroup $K$ with continuous inversion is an inverse semigroup, without loss of generality we may assume that $S_{\alpha_0}$ is a dense proper inverse subsemigroup of $K$.

Next, we put $S^{\prime}$ be the orthogonal sum of the family $\left\{S_{\alpha}\colon \alpha\in\mathscr{I}\setminus\{\alpha_0\}\right\}$ and the semigroup $K$. We determine a topology $\tau$ on $S^{\prime}$ in the following way.

First we observe if the orthogonal sum $T=\sum_{i\in\mathscr{J}}T_j$ is an inverse Hausdorff semitopological semigroup, then for every non-zero element $t\in T_j\subset T$ there exists an open neighbourhood $U(t)$ of $t$ in $T$ such that $U(t)\subseteq T_j^*$. Indeed, for every open neighbourhood $W(t)\not\ni 0$ of $t$ in $T$ there exists an open neighbourhood $U(t)$ of $t$ in $T$ such that $tt^{-1}\cdot U(t)\subseteq W(t)$. The neighbourhood $U(t)$ is requested.

We put that the bases of topologies at any point $s$ of $S\setminus S_{\alpha_0}$ and of $S^{\prime}\setminus K$ coincide in $S$ and in $S^{\prime}$, respectively. Also the bases at any point $s$ of subspace $K^*\subseteq S^{\prime}$ coincide with the base at the point $s$ of $K^*$. The following family determines the base of the topology $\tau$ at zero of the semigroup $S^{\prime}$:
\begin{equation*}
\begin{split}
  \mathscr{B}_0=\big\{U\subseteq S^{\prime}\colon & \mbox{there exist an element } V \mbox{~of the base at zero of the topology of~} S \\
    & \mbox{and an element~} W \mbox{~of the base at zero of the topology of~} K \mbox{~such that} \\
    & U\cap S^{\prime}\setminus K=V\cap S\setminus S_{\alpha_0}, \; U\cap K=W \mbox{~and~} U\cap S_{\alpha_0}=W\cap S_{\alpha_0}\big\}.
\end{split}
\end{equation*}
Simple verifications show that $(S^{\prime},\tau)$ is a Hausdorff semitopological inverse semigroup with continuous inversion and moreover $S$ is a dense proper inverse subsemigroup of $(S^{\prime},\tau)$, which contradicts the assumption of our theorem. The obtained contradiction implies the statement of the theorem.
\end{proof}

Theorem~\ref{theorem-4.5} implies the following corollary:

\begin{corollary}\label{corollary-4.6}
A primitive Hausdorff semitopological inverse semigroup $S$ is $($absolutely$)$ $H$-closed in the class of semitopological inverse semigroups with continuous inversion if and only if so is every its maximal subgroup $G$ with adjoined zero with an induced topology from $S$.
\end{corollary}

\begin{remark}\label{remark-4.7}
We observe that the statements of Theorems~\ref{theorem-4.3}, \ref{theorem-4.4} and ~\ref{theorem-4.5} hold for $H$-closed and absolute $H$-closed semitopological semilattices in the class of semitopological semilattices.
\end{remark}

\begin{theorem}\label{theorem-4.8}
An infinite semitopological semigroup of $\lambda\times\lambda$-matrix units $B_\lambda$ id $H$-closed in the class of semitopological semigroups if and only if the space $B_\lambda$ is compact.
\end{theorem}

\begin{proof}
Implication $(\Leftarrow)$ is trivial.

$(\Rightarrow)$. Suppose to the contrary that there exists a Hausdorff non-compact topology $\tau_B$ on the semigroup $B_\lambda$ such that $(B_\lambda,\tau_B)$ is an $H$-closed semigroup in the class of semitopological semigroups. By Lemma~2 of \cite{GutikPavlyk2005a} every non-zero element of $B_\lambda$ is an isolated point in $(B_\lambda,\tau_B)$. Then there exists an infinite open-and-closed subset $A\subseteq B_\lambda\setminus\{0\}$.

Then we have that at least one of the following cases holds:
\begin{itemize}
  \item[1)] there exist finitely many $i_1,\ldots,i_n\in\lambda$ such that if $(i,j)\in A$ then $i\in\{i_1,\ldots,i_n\}$;
  \item[2)] there exist finitely many $j_1,\ldots,j_n\in\lambda$ such that if $(i,j)\in A$ then $i\in\{j_1,\ldots,j_n\}$;
  \item[3)] cases 1) and 2) don't hold.
\end{itemize}

Suppose case 1) holds. Then there exists an element $i_0\in \{i_1,\ldots,i_n\}$ such that the set $\left\{(i_0,j)\colon j\in\lambda\right\}\cap A$ is infinite. We denote $A_{i_0}=\left\{(i_0,j)\in B_\lambda\colon (i_0,j)\in A\right\}$. It is obvious that $A_{i_0}$ is infinite subset of the semigroup $B_\lambda$. By Lemma~2 of \cite{GutikPavlyk2005a} every non-zero element of $B_\lambda$ is an isolated point in $(B_\lambda,\tau_B)$ and hence $A_{i_0}$ is an open-and-closed subset in the topological space $(B_\lambda,\tau_B)$. Since the left shift $l_{(i_0,i)}\colon B_\lambda\to B_\lambda\colon x\mapsto (i_0,i)\cdot x$ is a continuous map for any $i\in\lambda$, $A_{i}=\left\{(i,j)\in B_\lambda\colon (i_0,j)\in A\right\}$ is an infinite open-and-closed subset in $(B_\lambda,\tau_B)$ for every $i\in\lambda$. This implies that the set $B_\lambda\setminus\left\{A_{i_1}\cup\cdots\cup A_{i_k}\right\}$ is an open neighbourhood of the zero in $(B_\lambda,\tau_B)$ for every finite subset $\{i_1,\ldots,i_k\}\subset \lambda$.

Now. for every $i\in\lambda$ we put $a_i\notin B_\lambda$. We extend the semigroup operation from $B_\lambda$ onto the set $S=B_\lambda\cup\left\{a_i\colon i\in\lambda\right\}$ in the following way:
\begin{itemize}
  \item[$(i)$] $a_i\cdot a_j=a_i\cdot 0=0\cdot a_i=0$ for all $i,j\in\lambda$;
  \item[$(ii)$] $(s,p)\cdot a_i=
  \left\{
    \begin{array}{ll}
      a_s, & \hbox{if~} p=i;\\
      0, & \hbox{if~} p\neq i
    \end{array}
  \right.
  $ for all $(s,p)\in B_\lambda\setminus \{0\}$ and $i\in\lambda$;
  \item[$(iii)$] $a_i\cdot (s,p)=0$ for all $(s,p)\in B_\lambda\setminus \{0\}$ and $i\in\lambda$.
\end{itemize}
Simple verifications show that so defines binary operation on $S$ is associative, and hence $S$ is a semigroup.

Next, we define a topology $\tau_S$ on the semigroup $S$ in the following way. For every element $x\in B_\lambda$ we put that bases of topologies $\tau_B$ and $\tau_S$ at the point $x$ coincide. Also, for every $i\in\lambda$ we put
\begin{equation*}
    \mathscr{B}_S(a_i)=\left\{\{a_i\}\cup C_i\colon C_i \mbox{~is a cofinite subset of~} A_i\right\}
\end{equation*}
is a base of the topology $\tau_S$ at the point $a_i\in S$. It is obvious that $(S,\tau_S)$ is a Hausdorff topological space. The separate continuity of the semigroup operation in $(S,\tau_S)$ follows from the cofinality of the set $C_i$ in $A_i$ for each $i\in\lambda$. Therefore we get that the semitopological semigroup $(B_\lambda,\tau_B)$ is a dense proper subsemigroup of $(S,\tau_S)$, which contradicts the assumption of the theorem.

In case 2) the proof is similar.

Suppose that cases 1) and 2) don't hold. By induction we construct an infinite sequence $\left\{(x_i,y_i)\right\}_{i\in\mathbb{N}}$ in $B_\lambda$ in the following way. First we fix an arbitrary element $(x,y)\in A$ and denote $(x_1,y_1)=(x,y)$. Suppose that for some positive integer $n$ we construct the finite sequence $\left\{(x_i,y_i)\right\}_{i=1,\ldots,n}$. Since the set $A$ is infinite and cases 1) and 2) don't hold, there exists $(x,y)\in A$ such that $x\notin\{y_1,\ldots,y_n\}$ and $y\notin\{x_1,\ldots,x_n\}$. Then we put $(x_{n+1},y_{n+1})=(x,y)$.

Let $a\notin B_\lambda$. We put $T=B_\lambda\cup\{a\}$ and extend the semigroup operation from $B_\lambda$ onto $T$ in the following way:
\begin{equation*}
    a\cdot x=x\cdot a=a\cdot a=0, \; \mbox{~for every~} x\in B_\lambda.
\end{equation*}

Next, we define a topology $\tau_T$ on the semigroup $T$ in the following way. For every element $x\in B_\lambda$ we put that bases of topologies $\tau_B$ and $\tau_T$ at the point $x$ coincide. Also, we put
\begin{equation*}
    \mathscr{B}_T(a)=\left\{\{a\}\cup C\colon C \mbox{~is a cofinite subset of the set~} \left\{(x_i,y_i)\colon i\in\mathbb{N}\right\}\right\}
\end{equation*}
is a base of the topology $\tau_T$ at the point $a\in T$. It is obvious that $(T,\tau_T)$ is a Hausdorff topological space, the semigroup operation in $(T,\tau_T)$ is separately continuous, and $B_\lambda$ is a dense subsemigroup of $(T,\tau_T)$. This contradicts the assumption of the theorem.

The obtained contradictions imply the statement of our theorem.
\end{proof}

\begin{remark}\label{remark-4.9}
By Theorem~2~\cite{GutikPavlyk2005a} for every infinite cardinal $\lambda$ there exists a unique Hausdorff pseudocompact topology $\tau_c$ on the semigroup $B_\lambda$ such that $(B_\lambda,\tau_c)$ is a semitopological semigroup. This topology is compact and it is described in Example~1 of \cite{GutikPavlyk2005a}.
\end{remark}

\section*{Acknowledgements}

The author acknowledges Oleksandr Ravskyi for his comments and suggestions.


\end{document}